\documentclass[11pt,a4paper]{amsart}
\usepackage[margin=1.2in]{geometry}
\usepackage{bbm}
\usepackage{mathtools,amssymb,amsmath,amsopn,amsthm,caption,graphicx,centernot,cancel}
\usepackage{mathrsfs} 
\usepackage{tikz}
\usepackage{xcolor}
\usepackage{enumitem}
\usepackage{needspace}
\usepackage{algorithm,float}
\usepackage{algpseudocode}
\usepackage{verbatim}
\usepackage{hyperref}
\usepackage{tikz-cd}
\usepackage{stmaryrd} 
\usepackage{colonequals}

\usepackage[OT2,T1]{fontenc}
\hypersetup{
colorlinks=true,
linkcolor=blue,
filecolor=magenta,      
urlcolor=cyan,}
\usepackage[alphabetic, backrefs, lite]{amsrefs} 
\usetikzlibrary{matrix}

\def\Alphabet{A,B,C,D,E,F,G,H,I,J,K,L,M,N,O,P,Q,R,S,T,U,V,W,X,Y,Z}
\def\alphabet{a,b,c,d,e,f,g,h,i,j,k,l,m,n,o,p,q,r,s,t,u,v,w,x,y,z}
\def\endpiece{xxx}
\def\makeAlphabet[#1]{\expandafter\makeA#1,xxx,}
\def\makealphabet[#1]{\expandafter\makea#1,xxx,}
\def\makeA#1,{\def\temp{#1}\ifx\temp\endpiece\else%
\mkbb{#1}\mkfrak{#1}\mkbf{#1}\mkcal{#1}\mkscr{#1}\mkbs{#1}\expandafter\makeA\fi}%
\def\makea#1,{\def\temp{#1}\ifx\temp\endpiece\else\mkfrak{#1}\mkbf{#1}\mkbs{#1}\expandafter\makea\fi}%
\def\mkbb#1{\expandafter\def\csname bb#1\endcsname{\mathbb{#1}}}
\def\mkfrak#1{\expandafter\def\csname fr#1\endcsname{\mathfrak{#1}}}
\def\mkbf#1{\expandafter\def\csname b#1\endcsname{\mathbf{#1}}}
\def\mkcal#1{\expandafter\def\csname c#1\endcsname{\mathcal{#1}}}
\def\mkscr#1{\expandafter\def\csname s#1\endcsname{\mathscr{#1}}}
\def\mkbs#1{\expandafter\def\csname bs#1\endcsname{{\boldsymbol{#1}}}}
\def\makeop[#1]{\xmakeop#1,xxx,}
\def\mkop#1{\expandafter\def\csname #1\endcsname{{\mathrm{#1}}}} %
\def\xmakeop#1,{\def\temp{#1}\ifx\temp\endpiece\else\mkop{#1}\expandafter\xmakeop\fi}%
\def\makeup[#1]{\xmakeup#1,xxx,}
\def\mkup#1{\expandafter\def\csname #1\endcsname{{\mathrm{#1}\,}}} %
\def\xmakeup#1,{\def\temp{#1}\ifx\temp\endpiece\else\mkup{#1}\expandafter\xmakeup\fi}%
\makeAlphabet[\Alphabet]
\makealphabet[\alphabet]
\makeop[Hom,Tor,Ker,Im,Coker,holim,hocolim,dim,Sel,GL,H,R,Gal,inv,ord,lcm,gcd,res,C,D,Div,loc,Princ,loc,Supp,cor,sh,an,rk,tors,Vis,Cl,Reg,BSD]
\makeup[Spec,Proj,Spwf,Sp,Sh,Spf,Sch,id,dR,rig,cris,Tot,sm,an,Pic,NS,Br,Aut,NBSD]
  \makeatletter
    
    \@addtoreset{equation}{section}
  \makeatother

\newcommand*{\sheafhom}{\cH\kern -.5pt om}

\newcommand{\iso}{\simeq}

\newcommand\restr[2]{{
  \left.\kern-\nulldelimiterspace 
  #1 
  \vphantom{\big|} 
  \right|_{#2} 
  }}
  
  \newcommand{\ctp}{\mathrm{CT}}

\newcommand{\pair}[2]{\langle #1, #2\rangle}

\newcommand{\thmbreakheader}{\Needspace *{2\baselineskip}\item} 

\DeclareSymbolFont{cyrletters}{OT2}{wncyr}{m}{n}
\DeclareMathSymbol{\Sha}{\mathalpha}{cyrletters}{"58}

\theoremstyle{plain}
    \newtheorem{theorem}{Theorem}[section]
    \newtheorem{proposition}[theorem]{Proposition}
    
    \newtheorem{corollary}[theorem]{Corollary}

      \newtheorem{proposition-definition}[theorem]{Proposition-Definition}
      \newtheorem{conjecture}[theorem]{Conjecture}
      
\theoremstyle{definition}
    \newtheorem{definition}[theorem]{Definition}

    \newtheorem{remark}[theorem]{Remark}

\begin{document}

\author[Maistret, C.]{Celine Maistret}
\thanks{}
\address{Department of Mathematics, University of Bristol}
\email{\href{celine.maistret@bristol.ac.uk}{celine.maistret@bristol.ac.uk}}
\author[Shukla, H.]{Himanshu Shukla}
\address{Mathematisches Institut, Universit\"at Bayreuth}
\email{\href{Himanshu.Shukla@uni-bayreuth.de}{Himanshu.Shukla@uni-bayreuth.de}}

\title[]{On the factorization of twisted $L$-values and $11$-descents over $C_5$-number fields}

\begin{abstract}
We investigate the Galois module structure of the Tate-Shafarevich group of elliptic curves. 
For a Dirichlet character $\chi$, 
we give an explicit conjecture relating the ideal factorization of $L(E,\chi,1)$
to the Galois module structure of $\Sha(E/K)$, where $\chi$ factors through the Galois group of $K/\bbQ$.

We provide numerical evidence for this conjecture using the methods of visualization and $p$-descent. For the latter, we present a procedure that makes performing an $11$-descent over a $C_5$ number field practical for an elliptic curve $E/\bbQ$ with complex multiplication. We also expect that our method can be pushed to perform higher descents (e.g. $31$-descent) over a $C_5$ number field given more computational power. 
\end{abstract}

\maketitle

\tableofcontents
\section{Introduction}\label{introduction}

For an elliptic curve $E/\bbQ$ and a Dirichlet character $\chi$ factoring through a number field $K/\bbQ$ we investigate the connection between the Galois module structure of $\Sha(E/K)$ and the algebraic special $L$-value of the twisted $L$-function $L(E/\bbQ,\chi,1)$. We start by recalling the Birch and Swinnerton-Dyer conjecture (BSD), which links the size of Tate-Shafarevich groups of elliptic curves to the special value of their $L$-functions.

\subsection{Conjectural order of $\Sha(E/K)$}
Computing $\Sha(E/K)$ is a fundamental but hard problem.
Assuming its finiteness, the Birch and Swinnerton-Dyer (BSD) conjecture provides a formula for its size in terms of arithmetic invariants of the curve. 
Let $L(E/K,s)$ be the Hasse-Weil $L$-function associated to $E/K$. 
\begin{conjecture}[Birch and Swinnerton-Dyer Conjecture] \label{BSD}
    Let $E/K$ be an elliptic curve. Then 
\begin{enumerate}
    \item $\ord_{s=1} L(E/K,s) = \rk(E/K)$,
    \item the leading term of the Taylor series at $s=1$ of L(E/K,s) is given by 
    \begin{equation}\label{BSDQuo}
    \lim_{s \to 1} \frac{L(E/K,s)}{(s-1)^r}\cdot \frac{\sqrt{|\Delta_K|}}{\Omega_+(E)^{r_1+r_2}|\Omega_-(E)|^{r_2}}=\frac{\Reg_{E/K} |\Sha_{E/K}|C_{E/K}}{|E(K)_{\tors}|^2},
     \end{equation}
where $\Reg_{E/K}$ is the regulator of $E/K$, $\Delta_K$ is the discriminant of the field, $r$ is the order of the zero, $(r_1, r_2)$ is the signature of $K$, $\Omega_{\pm}$ are the periods of $E$ and $C_{E/K}$ is the product of Tamagawa numbers and other local fudge factors from finite places.
\end{enumerate}
\end{conjecture}

Following \cite{dew}, we will denote the right-hand side of \eqref{BSDQuo} by $\BSD(E/K)$. In practice, $L(E/K,s)$ is computable up to a ``reasonable'' approximation (see \cite{Tim_L_function}). Therefore Conjecture \ref{BSD}(2) can be used to reasonably approximate $|\Sha(E/K)|$ by computing the other quantities in $\BSD(E/K)$. This computation is referred to as computing the order of the \emph{analytic} Tate-Shafarevich group $\Sha_{\text{an}}(E/K)$ (see the L-function and Modular form database LMFDB \cite{lmfdb} for examples of such computations).

\subsection{Conjectural Galois module structure of $\Sha(E/K)$}\label{sec:conjGalSha}
If $K/\bbQ$ is a Galois extension and $E/\bbQ$ is an elliptic curve, then the $n$-descent sequence 
\begin{equation}\label{ndescent}
    0\to \frac{E(K)}{nE(K)}\to \Sel^{(n)}(E/K)\to \Sha(E/K)[n]\to 0
\end{equation}
 is an exact sequence of Galois modules. One can then ask if a BSD-like formula can be used to link the Galois module structure of $\Sha(E/K)$ to $L(E/\bbQ, \rho,s)$, where $\rho$ is an Artin representation associated to $K/\bbQ$. This is related to conjectures such as the Equivariant Tamagawa Number Conjecture, and can be formalized as follows. 

Following Notation 15 in \cite{dew} we define the Euler factor at a prime number $p$ of $L(E/\bbQ, \rho, s)$ to be 

\begin{equation}\label{euler_fact}
P_\ell(E,\rho, s)\coloneqq \det(\text{Id}-\text{Frob}_\ell^{-1} \ell^{-s} |(H_{p}^1(E)\otimes \rho)^{I_\ell}), 
\end{equation}
where $I_\ell$ denotes the inertia group at $\ell$ and $H_{p}^1(E) = H^1_{et}(E, \bbQ_{p})\otimes_{ \bbQ_{p}}\bbC$ for any embedding $\bbQ_{p} \hookrightarrow \bbC$ and any prime $p \ne \ell$.

The following definition presents the ``twisted'' analogue of the left-hand side of \eqref{BSDQuo}.
\begin{definition}[Definition 12 in \cite{dew}]\label{twistedLvalue}
For an elliptic curve $E/\bbQ$ and an Artin representation $\rho$ over $\bbQ$, we write 
    $$ \mathscr{L}(E,\rho)= \lim_{s \to 1} \frac{L(E,\rho,s)}{(s-1)^r}\cdot \frac{\sqrt{\frak{f}_{\rho}}}{\Omega_+(E)^{d^+(\rho)}|\Omega_-(E)|^{d^-(\rho)}w_{\rho}},
    $$
    where $r = \ord_{s=1}L(E,\rho,s)$ is the order of the zero at $s=1$, $\frak{f}_{\rho}$ and $w_{\rho}$ are the conductor and root number of $\rho$ respectively, and $d^{\pm}(\rho)$ are the dimensions  of the $\pm$-eigenspaces of complex conjugation in its action on $\rho$.
    \end{definition}

On the other hand, for $p$ an odd prime such that $(p,[K:\bbQ])=1$, $\Sha(E/K)[p]$ decomposes 
as an $\bbF_p$-representation of $\Gal(K/\bbQ)$:
\[\Sha(E/K)[p]\simeq \Sha(E/\bbQ)[p]\oplus
  \bigoplus\limits_\theta(\Sha(E/K)[p])_\theta,\]
where $(\Sha(E/K)[p])_\theta$ is the $\bbF_p$-subspace corresponding to the
non-trivial irreducible sub-representation $\theta$.

\begin{conjecture}\label{exact_question}
Let $p$ be an odd prime, $\chi$ a primitive Dirichlet character of square free order $d$, conductor $\frak{f}_\chi$, factoring through an extension $K^\chi/\bbQ$ with $\Gal(K^\chi/\bbQ)=\langle\tau\rangle$, and let $\mu_d$ denote the group of roots of unity in $\overline{\bbQ}$. Then there is an embedding $\iota:\bbQ(\zeta_d)\to \bbC$ so that for every elliptic curve $E/\bbQ$ with conductor $\frak{f}_E$ coprime to $\frak{f}_\chi$ satisfying \newline  
$\Sha(E/K^\chi)[p^\infty]=\Sha(E/K^\chi)[p]$,
     $E(K)[p^\infty]=E(\bbQ)[p^\infty]$,
    $E(K)[p^\infty]=E(\bbQ)[p^\infty]$, \newline 
     $v_p(C_{E/K})=v_p(C_{E/\bbQ})$, and 
     $L(E,\chi,1)\ne 0$,

we have
\[(\iota^{-1}(\sL(E,\chi)))_p=\prod\limits_\theta(h_\theta(\chi(\tau)^{d/d_\theta}),p),\]
where 
$(\sL(E, \chi))_p$ is the part above $p$ in the
   factorization of the principal ideal \newline $(\iota^{-1}(\sL(E,\chi)))\subset \bbZ[\zeta_d]$, and 
    for an irreducible representation $\theta$, $h_\theta$ is the minimal polynomial of the matrix associated to $\theta(\tau)$ in a chosen basis of $(\Sha(E/K)[p])_\theta$, $d_\theta$ is the order of the roots of $h_\theta$, and
    $\chi$ is regarded as a homomorphism $\chi: \Gal(K^\chi/\bbQ) \to \mu_d$.

\end{conjecture}

    We note that in the context of the above conjecture, the first non-trivial case to examine is that of a character of order 5 and 11-torsion in $\Sha(E/K)$. This is because for lower order characters and smaller torsion in $\Sha(E/K)$, there is no choice in the factorization of $\sL(E/\bbQ, \chi,s)$. 
We use the methods of visualization and $p$-descent to prove the following results towards Conjecture \ref{exact_question} in this first non-trivial case.

\begin{theorem}[Theorem \ref{th:mainVisu}]
Let $\chi$ be a primitive Dirichlet character of order $d$ with conductor $\frak{f}_\chi$. Let $E_1/\bbQ$ and $E_2/\bbQ$ be two $p$-congruent elliptic curves \emph{(}i.e., $E_1[p]\simeq E_2[p]$ as Galois modules\emph{)}. For $i\in\{1,2\}$, assume that $E_i$ has good ordinary reduction at $p$, $E_i[p]$ is irreducible, $(\frak{f}_{\chi},\frak{f}_{E_1}\frak{f}_{E_2})=1$, $L(E_1,\chi,1)\ne0$, $L(E_1,1)\ne0$, and $L(E_2,\chi,1)=0$. 

Furthermore, for each ideal $\frak{p}\subset \bbZ[\zeta_d]$ above $p$ with complex conjugate $\overline{\frak{p}}$,  assume that $v_\frak{p}(P_\ell(E_1,\chi, 1))=0$ for all $\ell\ |\ \frak{f}_{E_1}\frak{f}_{E_2}$ or $v_{\overline{\frak{p}}}(P_\ell(E_1,\chi, 1))=0$ for all $\ell\ |\ \frak{f}_{E_1}\frak{f}_{E_2}$.
Then \[(p)\ |\ (\sL(E_1,\chi))_p.\]
\end{theorem}

\begin{corollary}[Corollary \ref{co:mainVisu}]
If furthermore we assume that Conjecture \ref{BSD}(1) holds for $E_1$ and $E_2$ over $K^\chi$ and $\bbQ$ and that 
$\frac{E_2(K^\chi)}{pE_2(K^\chi)}\subset \Sha(E_1/K^\chi)[p]$ under the visualization map \emph{(}equation \eqref{vis_map}\emph{)}, then \[(p)\ |\ \prod_\theta(h_\theta(\chi(\tau)^{d/d_\theta}),p),\] where $h_\theta$, $d_\theta$, and $\tau$ are as in Conjecture \ref{exact_question}.
\end{corollary}

We use Theorem \ref{th:mainVisu}
and Corollary \ref{co:mainVisu} to prove Conjecture \ref{exact_question} for the elliptic curve with LMFDB-label ``9450du1''. 

However, this method limits the Galois module structure of $\Sha(E/K)[p]$ as explained in Remark \ref{rem:Vis_not_enough}. Therefore, to verify Conjecture \ref{exact_question} when other Galois module structures occur, we turn to the method of $p$-descent and develop a procedure to achieve an 11-descent over a $C_5$ number field for an elliptic curve over $\bbQ$ with complex multiplication (CM).

\begin{theorem}\label{Th:MainDesc}
    Assuming the Generalized Riemann's Hypothesis, Conjecture \ref{exact_question} holds for the elliptic curves
    with LMFDB labels
    ``5776.i1'', ``6400.a1'', ``7056.bg1'', 
    ``16641.g1'', ``57600.ch1'', ``90601.c1'', ``215296.c1'', ``461041.h1'', ``499849.d1''  with p=11 and all primitive Dirichlet characters factoring through $K^\chi= \bbQ(\zeta_{11})^+$.
    \begin{center}
    \begin{tabular}{|c|c|}
 \hline
           Label& Model\\
         \hline\hline
          5776.i1& $y^2=x^3-219488x+39617584$\\
          \hline
       6400.a1  & $y^2=x^3+x^2-333x+1963$ \\
       \hline
        7056.bg1 & $y^2=x^3-262395x +51731946$ \\
        \hline
        16641.g1 & $y^2+y=x^3-14311260x+20838446795$\\
        \hline
       57600.ch1  & $y^2=x^3-3000x+56000$\\
       \hline
       90601.c1 & $y^2+y=x^3-77916860x+ 264725453732$ \\
       \hline
       215296.c1 & $y^2=x^3+x^2-11213x+400939$\\
       \hline
       461041.h1 & $y^2+xy=x^3-x^2-17144962x+ 27327405657$\\
       \hline
       499849.d1 & $y^2+xy+y=x^3-x^2-18588135x+30849251024$\\
       \hline
    \end{tabular}
    \end{center}
\end{theorem}
To achieve Theorem \ref{Th:MainDesc} we developed the following algorithm to perform a $p$-descent over a $C_q$ extension $K/\bbQ$ under certain assumptions. See Algorithm \ref{alg:cap} for the complete procedure.

\begin{algorithm}[h]
\caption{}\label{alg:cap}
\begin{algorithmic}
\Require \thmbreakheader\begin{enumerate} \item Elliptic curve $E/\bbQ$ with complex multiplication by order $\cO$ and a prime $p$ such that $p$ splits in $\cO$.
\item A cyclic extension $K/\bbQ$ of prime degree $q$ such that $q\ |\ p-1$.
\end{enumerate}
\Ensure The minimal polynomial $h_\theta$ and the eigenspace $(\Sel_p(E/K))_\theta$.
\end{algorithmic}
\end{algorithm}

In theory, the descent procedure developed in this article works for  an arbitrary prime $p$ that splits in the order $\cO$ by which $E/\bbQ$ has CM, and when $[K:\bbQ]=q$ is a prime such that $q\ |\ p-1$. In practice however, one is limited by computer power. Although we expect that one can work with $p=31$ and $q=5$.   
    
\subsection{Related work}
Our work originates from Remark 43 in \cite{dew}, where predicting the factorization of $L(E/\bbQ,\chi,1)$ is mentioned in the following context.

Let $E/\bbQ$ be an elliptic curve and $K$ be a number field. Artin's formalism gives
\[L(E/K,s)=\prod\limits_\rho L(E/\bbQ, \rho, s),\]
where $\rho$ runs through the Artin representations of $K$. 
Conjecturally, one expects that \[\ord_{s=1}L(E/\bbQ,\rho, s)=\langle\rho, E(K)\otimes_{\bbZ}\bbC\rangle,\] 
where the right hand side is the multiplicity of $\rho$ in $E(K)\otimes_{\bbZ}\bbC$ as a representation of $\Gal(K/\bbQ)$ (\cite{rohrlich}). This suggests that the conjectural BSD-type formula for $L(E/\bbQ, \rho,s)$ must take into account the Galois module structure on the objects appearing in $\BSD(E/K)$.
The following theorem provides a partial answer. 

\begin{theorem}[Theorem 38 in \cite{dew}]\label{MainTh} 
  Let $E/\bbQ$ be an elliptic curve with conductor $\frak{f}_E$. Suppose Steven's Manin constant conjecture holds for $E/\bbQ$. 
Let $\chi$ be a non-trivial primitive Dirichlet character of order $d$ and  conductor coprime to $\frak f_E$. Then $\mathscr{L}(E,\chi) \in \bbZ[\zeta_d]$ and, if $L(E,\chi, 1) \ne 0$, then furthermore 
  \[\zeta\cdot \mathscr{L}(E,\chi) \in \bbR, \qquad \text{for } \zeta = \chi(\mathfrak f_E)^{\frac{d+1}{2}}\sqrt{\chi(-1)w_E}.
  \]
If $\rk(E/\bbQ)=0$ and the BSD conjecture holds for $E$ over $\bbQ$ and $K^{\chi}$ with $d$ squarefree, then 

  \begin{equation}\label{norm_bsd}
N_{\bbQ(\zeta_d)^+/\bbQ}(\zeta\cdot \mathscr{L}(E,\chi))=\pm \frac{|E(\bbQ)_{\tors}|}{|E(K^{\chi})_{\tors}|}\sqrt{\frac{|\Sha_{E/K^{\chi}}|\prod_vc_v(E/K^{\chi})}{|\Sha_{E/\bbQ}|\prod_pc_p(E/\bbQ)}}.
\end{equation}
If moreover $d$ is odd and $\BSD(E/K^{\chi})=\BSD(E/\bbQ)$ then $\mathscr{L}(E,\chi) 
  = \zeta^{-1}u$ for some unit $u \in \mathcal{O}^{\times}_{\bbQ(\zeta_d)^+}$.
\end{theorem}

\begin{remark}
We note that if one assumes the conductor of $\chi$ to be coprime to that of $E$, the integrality of $\sL(E,\chi)$ above follows from Theorem 2 in \cite{WW}, hence does not depend on the Birch and Swinnerton-Dyer conjecture. 
\end{remark}

The following example (Example 42 in \cite{dew}) exhibits the limitations of the above theorem regarding the Galois module structure of $\Sha(E/K)$. 

Consider the elliptic curves
$$
E_1/\bbQ : y^2 +y=x^3-x^2-1 \quad \text{ and } \quad E_2/\bbQ : y^2 + xy = x^3+x^2-3x-4,
$$ with conductors 291 and 139 respectively. Let $\chi$ be the primitive character of order 5 and conductor 31 defined by $\chi(3) = \zeta_5^3$. For both curves, we have 
    \[
|E_i(\bbQ)| = |E_i(K^{\chi})| = |\Sha_{E_i/\bbQ}| = \prod_pc_p(E_i/\bbQ) = \prod_v
c_v(E_i/K^\chi)=1
\]
and $|\Sha_{E_i/K^{\chi}}|=11^2.$
One computes the following ideal factorization
    \[ 
      (\sL(E_1,\chi)) = \frak p_1 \overline{\frak p_1} \quad \text{ and } (\sL(E_2,\chi)) = \frak p_2 \overline{\frak p_2},
    \]
where $\frak p_1, \overline{\frak p_1}, \frak p_2, \overline{\frak p_2}$ are distinct ideals above 11, with $\overline{\frak p_i}$ the image of $\frak p_i$ under complex conjugation.

Theorem \ref{MainTh} predicts that the ideals $\sL(E_i, \chi)$ must be either $\frak p_1 \overline{\frak p_1}$ or $ \frak p_2\overline{\frak p_2}$, but  cannot be used to decide which factorization it is. This prompted Conjecture \ref{exact_question}.




\subsection{Structure of the paper}
Section \ref{Sec_vis} presents the background for the method of visualization and use this method to detail numerical evidence towards Conjecture \ref{exact_question} for the elliptic curve with LMFDB-label ``9450.du1''. Section \ref{Sec:p-descent} presents the necessary background to perform a $p$-descent on an elliptic curve (Section \ref{subbackground}) and present a refinement to obtain a practical method for elliptic curves with complex multiplication (Sections \ref{subpdescentcm}, \ref{classgp_unitgp} and \ref{subgaloismodule}). 
It also presents some optimizations in order to cut down the search space using Cassels-Tate pairing (Proposition \ref{ctp_consequence}). This section ends with the detail of the numerical verification for the elliptic curve ``7056.bg1''.
All codes are available at \cite{github} where all the examples can be found. 

\subsection{Acknowledgements}
Both authors were supported by the first author's Royal Society Dorothy Hodgkin Fellowship. We wish to thank Vladimir Dokchitser, Michael Stoll, Aurel Page and Claus Fieker for helpful discussions and suggestions, and Timo Keller for pointing out the reference \cite{Gb_Vt}, which prompted Theorem \ref{th:mainVisu}.

\section{Numerical evidence using visualization}\label{Sec_vis}
\subsection{Background}
We provide an outline of the visualization method which will be used to provide evidence towards Conjecture \ref{exact_question}. We refer the reader to \cite{vis7fisher} for more details. 

Let $p$ be a prime number and $E_1/K$ and $E_2/K$ be $p$-congruent elliptic curves,
i.e., there exists an isomorphism 
$\sigma:E_1[p]\stackrel{\simeq}{\to} E_2[p]$ of Galois modules.
The idea of visualization is to use the above isomorphism to identify the elements of $\Sha(E_1/K)[p]$ as a 
subgroup of $\frac{E_2(K)}{pE_2(K)}$. 

Let $\Delta$ be the image of $E_1[p]$ under the diagonal embedding $P\mapsto
(P,\sigma(P))$ in $E_1\times E_2$ and let $A\coloneqq E_1\times E_2/\Delta$. 
Define the \emph{visible part} of $E_i$ in $A$ as
\[\Vis_A(E_i)\coloneqq \ker(\H^1(G_K, E_i)\to \H^1(G_K, A)).\]
The natural map 
$E_i\to A/E_j=E_i$ has kernel isomorphic to $\Delta$ for $i \in \{1,2\}$ and gives the following commutative diagram.
\[
  \begin{tikzcd}
    0\arrow[r] &\Delta \arrow[r]\arrow[d] &E_i\arrow [r]\arrow[d] &E_i
    \arrow[r]\arrow[d,equal] &0\\
    0\arrow[r] &E_j\arrow[r] &A\arrow[r] &E_i \arrow[r] &0. 
  \end{tikzcd}
\]
Applying Galois cohomology, one obtains the following exact sequence for 
$i\ne j \in \{1,2\}$
\begin{equation}\label{vis_map}
0\to \frac{A(K)}{E_1(K)+E_2(K)}\to \frac{E_i(K)}{pE_i(K)}\to \Vis_A(E_j)\to 0,
\end{equation}
which leads the following theorem.
\begin{theorem}[Theorem \cite{vis7fisher}, Theorem 2.2]\label{fisher}
  Let $E_i$, $A$ and $\Delta$ be as above and assume that
  $\frac{E_i(K)}{pE_i(K)}=0$.  If
  \begin{enumerate}
  \item all Tamagawa numbers of $E_i$ are coprime to  $p$; and
  \item for all places $v$ above $p$ of $K$, $A$ has good reduction and 
    $e(K_v/\bbQ_p)\leq p-1$,
  \end{enumerate}
  then $\frac{E_j(K)}{pE_j(K)}=\Vis_A(E_i)\cap \Sha(E_i)[p]$.
\end{theorem}


We now prove that $(p)\ |\ (\sL(E_1,\chi))_p$, for two $p$-congruent elliptic curves $E_1/\bbQ$ and $E_2/\bbQ$  satisfying certain conditions. 

\begin{theorem}\label{th:mainVisu}
Let $\chi$ be a primitive Dirichlet character of order $d$ with conductor $\frak{f}_\chi$. Let $E_1/\bbQ$ and $E_2/\bbQ$ be two $p$-congruent elliptic curves \emph{(}i.e., $E_1[p]\simeq E_2[p]$ as Galois modules\emph{)}. For $i\in\{1,2\}$, assume that $E_i$ has good ordinary reduction at $p$, $E_i[p]$ is irreducible, $(\frak{f}_{\chi},\frak{f}_{E_1}\frak{f}_{E_2})=1$, $L(E_1,\chi,1)\ne0$, $L(E_1,1)\ne0$, and $L(E_2,\chi,1)=0$. 

Furthermore, for each ideal $\frak{p}\subset \bbZ[\zeta_d]$ above $p$ with complex conjugate $\overline{\frak{p}}$,  assume that $v_\frak{p}(P_\ell(E_1,\chi, 1))=0$ for all $\ell\ |\ \frak{f}_{E_1}\frak{f}_{E_2}$ or $v_{\overline{\frak{p}}}(P_\ell(E_1,\chi, 1))=0$ for all $\ell\ |\ \frak{f}_{E_1}\frak{f}_{E_2}$.
Then \[(p)\ |\ (\sL(E_1,\chi))_p.\]
\end{theorem}

\begin{proof}
Let $\frak{p}$  be a prime ideal of $\bbZ[\zeta_d]$ above $p$. Let $x\mapsto x_{\frak{p}}$
be the localization map at $\frak{p}$ corresponding to an embedding of $\sigma_{\frak{p}}: \bbZ[\zeta_d]\hookrightarrow\bbZ_p[\zeta_d]$.
Let \[
\Sigma:=\{\ell\ |\ \ell \text{ a prime with }\ell\ne p\text{ and } \ell\ |\ \frak{f}_{E_1}\frak{f}_{E_2}\}.
\]
    Let $\cL_p^\Sigma(E_i,\chi,T)$ be the $p$-adic L-function as defined in \cite{Gb_Vt}. Then
    \[\cL_p^\Sigma(E_i,\chi,0)=\left(G(\overline{\chi})\frac{L_\Sigma(E_i,\chi,1)}{\Omega_{\pm}(E)}\right)_\frak{p},\]
    where $G(\overline{\chi})$ is the Gauss sum of $\overline{\chi}$, $\Omega_{\pm}(E)$ is $\Omega_+(E)$ when $\chi(-1)=1$ and $\Omega_-(E)$ otherwise, and $L_\Sigma(E_i,\chi,s)=\prod\limits_{\ell\notin \Sigma}P_\ell(E_i,\chi,s)$. Since $\chi$ is a primitive Dirichlet character, Definition \ref{twistedLvalue} gives
    \[\sL(E_i,\chi)=\frac{\sqrt{\frak{f}_\chi} L(E_i,\chi,1)}{w_\chi|\Omega_{\pm}(E_i)|}=\frac{G(\overline{\chi})L(E_i,\chi,1)}{\sqrt{\chi(-1)}|\Omega_{\pm}(E_i)|}=\chi(-1)\frac{G(\overline{\chi})L(E_i,\chi,1)}{\Omega_{\pm}(E_i)}.
    \]
Therefore, using \cite[Theorem 3.10]{Gb_Vt} we obtain that there is a $p$-adic unit $u\in \bbZ_p[\zeta_d]^\times$ and $r\in \bbZ_p[\zeta_d]$ so that 
\begin{align*}
    (\sL(E_1,\chi))_\frak{p} &= \chi(-1)\cL_p^\Sigma(E_1,\chi,0)\left(\prod\limits_{\ell\in\Sigma}P_\ell(E_1,\chi,1)\right)_\frak{p}\\
    &=\chi(-1)u(\cL_p^\Sigma(E_2,\chi,0)+pr)\left(\prod\limits_{\ell\in\Sigma}P_\ell(E_1,\chi,1)\right)_{\frak{p}}\tag{$\cL_p(E_2,\chi,1)=0$ and $p\ \nmid\ d$}\\
    &=(u\chi(-1))\cdot(pr)\cdot \left(\prod\limits_{\ell\in\Sigma}P_\ell(E_1,\chi,1)\right)_\frak{p}.
\end{align*}
Our assumption that $p\ \nmid \ P_\ell(E_1,\chi, 1)$ for all $\ell \in\Sigma$ for at least one of the localizations $\frak{p}$ and $\overline{\frak{p}}$ implies that $v_\frak{p}(\sL(E_1,\chi))\geq 1$ or $v_{\overline{\frak{p}}}(\sL(E_1,\chi))\geq 1$. Together with Theorem \ref{MainTh}, this gives that  $v_\frak{p}(\sL(E_1,\chi))\geq 1$ and $v_{\overline{\frak{p}}}(\sL(E_1,\chi)) \geq 1$.
\end{proof}

\begin{corollary}\label{co:mainVisu}
If furthermore we assume 
that Conjecture \ref{BSD}(1) holds for $E_1$ and $E_2$ over $K^\chi$ and $\bbQ$ and that 
$\frac{E_2(K^\chi)}{pE_2(K^\chi)}\subset \Sha(E_1/K^\chi)[p]$ under the visualization map \emph{(}equation \eqref{vis_map}\emph{)}, then \[(p)\ |\ \prod_\theta(h_\theta(\chi(\tau)^{d/d_\theta}),p),\] where $h_\theta$, $d_\theta$, and $\tau$ are as in Conjecture \ref{exact_question}.
\end{corollary}

\begin{proof}
First assume that 
 the $\bbQ$-vector space
 $\frac{E_2(K^\chi)}{E_2(\bbQ)}\otimes\bbQ$ splits as $\varphi(d)$-dimensional irreducible vector spaces and 
\[\left(\frac{E_2(K^\chi)}{pE_2(K^\chi)}\right)/\left(\frac{E_2(\bbQ)}{pE_2(\bbQ)}\right)\simeq \frac{E_2(K^\chi)\otimes \bbF_{p}}{E_2(\bbQ)\otimes \bbF_{p}}, 
\] where $\varphi$ denotes the Euler totient function.  In particular, every non-trivial irreducible $
\Gal(K^\chi/\bbQ)$-subrepresentation of $\left(\frac{E_2(K^\chi)}{pE_2(K^\chi)}\right)/\left(\frac{E_2(\bbQ)}{pE_2(\bbQ)}\right)$ corresponding to the irreducible factors of the $d^{\text{th}}$ cyclotomic polynomial $\Phi_d \mod p$ occurs with same multiplicity. The same holds for the image of $\left(\frac{E_2(K^\chi)}{pE_2(K^\chi)}\right)/\left(\frac{E_2(\bbQ)}{pE_2(\bbQ)}\right)$ in  $\frac{\Sha(E_1/K^\chi)[p]}{\Sha(E_1/\bbQ)[p]}$ under the visualization map. Therefore $(p)\ |\ \prod\limits_\theta (h_\theta(\chi(\tau)),p)$. 

In the case when there is a subfield $\bbQ\ne F\subsetneq K^\chi$ such that $[F:\bbQ]=d'$ and  $\dim_\bbQ\left(\frac{E_2(F)}{E_2(\bbQ)}\right) \ne 0$, then one can restrict $\chi$ to $\Gal(F/K)\simeq \langle\tau^{d/d'}\rangle$ and use the same argument as above for $F$. This gives $(p)| \prod\limits_\theta h_\theta(\chi(\tau)^{d/d'},p)$, where $\theta$ runs through all non-trivial irreducible representations of $\Sha(E_1/F)[p]$.
\end{proof}

\begin{remark}
    If $d$ is a prime $q$ such that $\bbF_p$ or $\bbF_{p^2}$ does not contain $q^{\text{th}}$ roots of unity, $\chi(\ell)\ne 1$ for all $\ell \in \Sigma$ (as in Theorem \ref{th:mainVisu}), and the hypothesis of theorem \ref{fisher} is satisfied,
    then $(p)\ |\ (\sL(E_1,\chi))$ if and only if $(p)\ |\ \prod\limits_\theta(h_\theta(\chi(\tau)),p)$. An example would be $q=61$ with $K^\chi$ as degree 61 subfield of $\bbQ(\zeta_{367})$ and $p=11$.
\end{remark}
 \subsection{Numerical evidence} \label{evidence_vis}We present an example using the curve with LMFDB-label ``9450.du1''. More examples can be found in \cite{github}, together with the codes.

Consider the case $p=11$, $K^\chi := \bbQ(\zeta_{11})^+$ and 
$E_1\coloneqq$ ``9450.du1''. Let $\chi$ be a primitive character with field of definition $K^\chi$, conductor $\frak{f}_\chi=11$ and root number $w(\chi)=1$. We verify that Conjecture \ref{exact_question} holds for $E_1$ and all characters of $K^\chi$.

 $E_1$ can be given by the Weierstrass model $$E_1 : y^2+xy+y=x^3-x^2-271946810x-1733074051463.$$ We first find an 11-congruent curve, say $E_2 :  y^2+xy+y=x^3-x^2-5x+7$ (``9450.dr1'').

Keeping notation as in Theorem \ref{MainTh}, we check using Magma, that $E_1$ satisfies its conditions  and that the pair $(E_1, E_2)$ satisfies the conditions of Theorem \ref{fisher}. We find that

 $\rk(E_1/\bbQ)=\rk(E_2/\bbQ)=0$,
 $\rk(E_1/K^\chi)=0$ and $\rk(E_2/K^\chi)=4$,
 $|\Sha_{\an}(E_1/\bbQ)| = |\Sha_{\an}(E_2/\bbQ)|=1$, 
 $(c_v(E_i),11)=1$, for $i\in\{1,2\}$, and $v$ a place of $K^{\chi}$ or $\bbQ$. Furthermore, \[\prod\limits_{M_{K^\chi}}c_v(E_i)=\prod\limits_{M_{\bbQ}}c_v(E_i)\] for $i\in\{1,2\}$, where $M_{K^\chi}$ is the set of finite places of $K^\chi$,
 $|\Sha_\an(E_1/K^\chi)|= 2^4\cdot5^2\cdot11^4\cdot59^2$, and $|\Sha_\an(E_2/K^\chi)|=1$,
 and $(E_1)_\tors(K^\chi)=(E_2)_\tors(K^\chi)=1$. 

By Theorem \ref{MainTh} we have 
$N_{\bbQ(\zeta_5)^+/\bbQ}(\sL(E_1,\chi))=2^2\cdot5\cdot11^2\cdot59$. Using the fact that $2$ is inert, and ideals above $5$ and $59$ are principal in $\cO_{\bbQ(\zeta_5)^+}$, we compute that  
$(\sL(E_1, \chi))_{11} = (11)$. 

In fact one can use Theorem \ref{th:mainVisu}
directly. One computes that $\Sigma=\{2,3,5,7\}$ and that $3,5$ are the primes of additive reduction and $2,7$ are the primes of split multiplicative reduction. Note that in Theorem \ref{th:mainVisu} one can ignore the primes of additive reduction. We have $P_2(E_1,\chi)=(1-\chi(2)/2)^{-1}$ and $P_7(E_1,\chi)=(1-\chi(7)/7)^{-1}$. Since $2,7\notin \mu_5\subset \bbF_{11}$, we obtain that $v_{\frak{p}_i}(\sL(E_1,\chi))=1$ for $i\in\{1,    2\}$, where $\frak{p}_1$ and $\frak{p}_2$ are the prime ideals of $\cO_K$ above $11$ that are not complex conjugates of each other.

On the other hand, using the function \texttt{MordellWeilGroup()} we find that  
$E_2(K^{\chi})$ is generated by 
\begin{align*}
 \{
        &(\alpha^4 - 3\alpha^2 + 1:
        2\alpha^4 - \alpha^3 - 6\alpha^2 + 2\alpha + 1:1),\\
        &(\alpha^4 + 5\alpha^3 - \alpha^2 - 12\alpha + 1: 7\alpha^4 + 15\alpha^3 - 16\alpha^2 - 28\alpha + 9: 1),\\
        &(-2\alpha^4 - 3\alpha^3 + 12\alpha^2 + 3\alpha - 4:
        20\alpha^3 - 28\alpha^2 - 25\alpha + 11:1),\\
    &(-3\alpha^4 + 6\alpha^3 + 4\alpha^2 - 16\alpha + 12:
        -12\alpha^4 + 29\alpha^3 + 21\alpha^2 - 82\alpha + 37: 1)
\},
\end{align*}
where $\alpha$ is a root of the polynomial $x^5 - x^4 - 4x^3 + 3x^2 + 3x - 1$. 

\vspace{.1in}
Viewing $\frac{E_2(K^{\chi})}{11E_2(K^{\chi})}$ as a $4$-dimensional $\bbF_{11}$-vector space, we obtain a representation 
$$\rho:\Gal(K^{\chi}/\bbQ)\coloneqq \langle \tau\rangle\to \GL_4(\bbF_{11}),$$
where $\tau: \alpha\mapsto -\alpha^2+2$.
Let $M_\tau$ be the matrix representing $\rho(\tau)$ obtained using the above generators of $E_2(K^\chi)$ as basis vectors. Then
\[
M_{\tau} = \begin{pmatrix}
        -1 &-1  &0  &0\\
        0  &0 &-1  &0\\
        1  &1  &0 &-1\\
        0 &-1  &0  &0
    \end{pmatrix}.
\]
The eigenvalues of $M_{\tau}$ are $\{3,4,5,9\}$, i.e., each of the non-trivial eigenspaces of $\rho$ is of dimension 1.
Therefore, ``expectedly'', $11$ must exactly divide the ideal generated by $\sL(E_1,\chi)$ in $\bbZ[\zeta_5]$, which is the case.

\begin{remark}\label{rem:Vis_not_enough}
If Conjecture \ref{exact_question}
is true then using Theorem \ref{fisher}  to completely visualize $\Sha(E/K)[p]$ forces $(\sL(E,\chi))_{p}=(p)^n$ for some $n\in \bbN$. To circumvent the above, one needs to falsify the conditions in Theorem \ref{fisher}, which is certainly challenging as it requires either to construct a $p$-congruent elliptic curve $E_2$ such that $A$ has bad reduction at $p$, or such that $p$ divides at least one of the Tamagawa numbers, or to construct a higher dimensional abelian variety to visualize $\Sha(E/K^{\chi})[p]$. All of the above are hard, and instead we use the method of $p$-descent to verify Conjecture \ref{exact_question} for some elliptic curves when $q=5$ and $\frac{|\Sha(E/K)[11]|}{|\Sha(E/\bbQ)[11]|}=11^2$.

\end{remark}

\section{Numerical evidence using $p$-descent}\label{Sec:p-descent}

\subsection{Background}\label{subbackground}
Let $p$ be an odd prime, let $K$ be a number field, $\overline{K}$ its algebraic closure and $G_K$ its absolute Galois group. 
 Applying Galois cohomology to 
 \[0\to E(\overline{K})[p]\to E(\overline{K})\stackrel{\cdot p}{\to} E(\overline{K})\to 0,\]
 we obtain
 \[\begin{tikzcd}
0\arrow[r] &\frac{E(K)}{pE(K)} \arrow[r] \arrow[d] &\H^1(G_K, E(\overline{K})[p]) \arrow[r] \arrow[d,"\prod\limits_v\res_v"] \arrow[rd,"\alpha"] & \H^1(G_K, E(\overline{K}))[p] \arrow[r]\arrow[d,"\prod\limits_v\res_v"] &0\\
0\arrow[r] &\prod\limits_v\frac{E(K_v)}{pE(K_v)} \arrow[r] &\prod\limits_v\H^1(G_{K_v}, E(\overline{K_v})[p]) \arrow[r]  & \prod\limits_v\H^1(G_{K_v}, E(\overline{K_v}))[p] \arrow[r] &0.
\end{tikzcd}
\]
The $p$-Selmer group and the Tate-Shafarevich group are defined as 
$\Sel_p(E/K)\coloneqq \ker(\alpha)\subset\H^1(G_K,E(\overline{K})[p])$ and $\Sha(E/K)\coloneqq \ker(\H^1(G_K, E(\overline{K}))\to \prod\limits_v\H^1(G_{K_v},E(\overline{K_v}))$, respectively.
In order to obtain a concrete description of the $p$-Selmer group and ultimately information about $\Sha(E/K)[p]$, one needs a good enough
description of $\H^1(G_K, E(\overline{K})[p])$. We follow  \cite{schaefer-stoll-p-descent} for such a description.

In general, let $\Delta$ be a $G_K$-set, 
$M$ a $G_K$-module and define
$M^\Delta\coloneqq \{\mathrm{Maps}: \Delta \to M\}$. 
 $M^\Delta$ has a
$G_K$-module structure with a natural action of $G_K$ given by 
\[g\cdot m \coloneqq P\mapsto gm(g^{-1}P), \quad \text{ for } g \in G_K, m \in M^\Delta \text{ and } P\in \Delta.\]

Denote by $L_\Delta$ the \'{e}tale algebra  associated with $\Delta$. This is the subalgebra 
$\overline{K}^\Delta$ fixed by $G_K$, i.e. the maps $m$ such that $g(m(P))= m(gP)$.
To every orbit $\Delta_i$ of $\Delta$ under the $G_K$-action, choose a representative $P_i$ and associate to $\Delta_i$ the field of definition of $P_i$, 
$K(P_i)$. Then $L_{\Delta_i}$ is isomorphic to $K(P_i)$ where $l\in K(P_i)$ represents the map $P_i^g \mapsto l^g$. 

\medskip

In the case where $\Delta\coloneqq E[p]\setminus\{\infty\}$, we let $L=L_\Delta$ be the associated 
\'{e}tale algebra as described above. Identify $E[p]\coloneqq \pair{P}{Q}$ with $\bbF_p^2$ and let $\Delta^+$ be the $G_K$-set of 
lines passing through the origin in $E[p]$. Each element of $\Delta^+$ is the subgroup generated by $iP+Q$ for $0\leq i\leq p-1$ and $P$. Let $L^+$ be the \'{e}tale algebra
associated with $\Delta^+$. The natural map $\Delta \to \Delta^+$
\[
  P \mapsto \langle P\rangle\coloneqq \text{line passing through $P$ and $\infty$},
\]
defines an embedding $L^+ \to L$ as $K$-algebras. Using this we identify $L^+$ with a
subalgebra of $L$. Generically, $\Delta$ has a single orbit and hence $L$ is a field. When this is not the case, the situation is simpler and therefore we assume that $L$ is a field, so that we have $[L:L^+]= p-1$. We fix a representative $P$ of $\Delta$ that fixes an embedding of $L\hookrightarrow\overline{K}$.

Let $w: E[p]\to \mu_p^\Delta$ be the map induced by the Weil-pairing $Q\mapsto e(Q,\_)\in
\mu_p^\Delta$, where $\mu_p$ denotes the $p^{th}$-roots of unity. The compatibility of the Weil-pairing with Galois action implies that $w$ is a well-defined
morphism of $G_K$-modules. Furthermore, a generalized version of Hilbert's theorem 90 implies that $\H^1(G_K, \mu_p^\Delta)$ can be identified with $L^\times/(L^\times)^p$.

The following theorem due to Schaefer and Stoll embeds $\H^1(G_K, E[p])$ inside $L^\times/(L^\times)^p$.

\begin{theorem}[Lemma 5.2 in \cite{schaefer-stoll-p-descent}]\label{gen_p_descent}
  Let $L^+\subset L\subset \overline{K}$ be as above, and $\Gal(L/L^+)=\langle
  \gamma\rangle\stackrel{\phi}{\simeq} \bbF_p^\times$, where the last isomorphism is given by $P^\gamma = \phi(\gamma) P$. Then the push forward of~$w$, $w_*: \H^1(G_K, E[p])\to
  \H^1(G_K, \mu_p^\Delta)$ is injective and 
  $\Im(w_*)\subset L^\times/(L^\times)^p$ is contained in the subgroup
  \[\left(\frac{L^\times}{(L^\times)^p} \right)^{(1)}\coloneq\{l\in L^\times\ |\
  l^\gamma=l^{\phi(\gamma)} \mod (L^\times)^p\}.\]
\end{theorem}
We note that the subgroup $\left(\frac{L^\times}{(L^\times)^p} \right)^{(1)}$ in which sits the image of $\H^1(G_K, E[p])$ is infinite. However, if we let $S$ be the set of places of bad reduction of $E/K$ as well as places above $p$, and $R(L,S;p)$
be the subgroup of $L^\times/(L^\times)^p$ cut out by elements $l\in L^\times$ such that $L(l^{1/p})/L$ is unramified everywhere outside $S$, then the local triviality condition on $\Sel_p(E/K)$ implies that \[\Sel_p(E/K)\subset R(L,S;p),\] where $R(L,S;p)$ is now finite.

The following result of Poonen and Schaefer allows us compute $R(L,S;p)$ in terms of $S$-class groups, $\Cl_S(L)$, and $S$-unit groups, $U_S(L)$, of $L$.
\begin{proposition}[Proposition 12.6 in \cite{poonen-schaefer}] \label{S-units_PS}
Let $L$, $S$ be as above. Then the sequence
\[1\to \frac{U_S(L)}{(U_S(L))^p}\to R(L,S;p)\to
\Cl_S(L)[p]\to 0.\]    
is exact. 
\end{proposition}

We now argue that $S$ can be reduced to places that completely split in $L$. 
Let $\frak{r}$ be a prime ideal of $L^+$ above a prime $r\in \bbZ$ and $\frak{r}_1,\ldots \frak{r}_k$
be the prime ideals of $L$ above $\frak{r}$.
The following proposition gives a condition for primes above $\frak{r}$ to be removed from $S$.
\begin{proposition} \label{S-empty}
Let $l\in L^\times$ represent an element of $\Sel_p(E/K)$ such that $(l)$ is a non-trivial ideal. Write the part of $(l)$ above $\frak{r}$ as $\prod\limits_{i=1}^k\frak{r}_i^{\alpha_i}$ for some $\alpha_i\in \bbZ$. Then $k=p-1$.
\end{proposition}
\begin{proof}
    Recall that $\Gal(L/L^+)=\langle \gamma\rangle\stackrel{\phi}{\simeq}\bbF_p^\times$.
    Since $l$ is an element of $\Sel_p(E/K)$, we have the following equivalence modulo $p$-powers of ideals
    \[\prod\limits_{i=1}^k(\frak{r}_i^\gamma)^{\alpha_i}=\prod\limits_{i=1}^k\frak{r}_i^{\phi(\gamma)\alpha_i}.\]
    This gives $\alpha_i= \phi(\gamma)\alpha_{\gamma\cdot i}$, where the action of $\gamma$ on the indices $i$ is induced from the action of $\gamma$ on $\frak{r}_i$. It follows that $\phi(\gamma)^k=1$ and therefore $k=p-1$.
\end{proof}

\begin{remark}
It follows from Proposition \ref{S-units_PS} that the $p$-Selmer group computation requires us to compute $p$-torsion in class group of $L$, and a subgroup of the unit group of $L$ with index coprime to $p$. In our particular setup where $p=11$, $[K:\bbQ]=5$ and $[L:K]=p^2-1=120$, $[L:\bbQ]=600$. Computing class groups and unit groups in such high degree extensions is currently out of reach (even after assuming GRH).
\end{remark}

\subsection{Elliptic curves with complex multiplication}\label{subpdescentcm}
Thanks to the following result one can reduce the degree of the field extension over which we need to compute the class groups and unit-groups by using elliptic curves with complex multiplication by $\cO$ in an imaginary quadratic number field $K'$.

\begin{corollary}[Theorem 8.1 in \cite{schaefer-stoll-p-descent}]\label{p-descent_cm}
Let $E/K$ be an elliptic curve with potential complex multiplication by $\cO\subset K'$ as above.
Assume that $p$ splits in $\cO$, i.e., $p\cO=\frak{p}_1\frak{p}_2$. Let $\Delta_i\coloneqq E[\frak p_i]\setminus \{\infty\}$. Then Theorem \ref{gen_p_descent} holds if $\Delta$ and $\Delta^+$ from before are replaced by the $G_K$-sets
  $\Delta_1\sqcup \Delta_2$ and $\{\Delta_1,\Delta_2\}$ respectively. 
\end{corollary}
\begin{remark} \label{deg_too_high}
With $\Delta$ and $\Delta^+$ as in the above corollary, $[L:K]=2(p-1)=20$ and we would need to compute class groups and unit groups of a degree 100 number field. Though, we have cut down the degree from 600 to 100, the class group and unit group computations are still infeasible if done naively. The content of \S\ref{classgp_unitgp} is to develop a procedure using \cite{norm_rel} that cuts down the class groups and unit groups computation to that of degree 20 number fields. 
\end{remark}

\subsection{Computing S-units and S-class groups} \label{classgp_unitgp}~

We briefly sketch the strategy and then concretely lay out the procedure we used. 
Let $E/\bbQ$ be an elliptic curve with complex multiplication by an order $\cO$ of an
imaginary quadratic field $K'/\bbQ$. Let $K$ be a number field with 
$\Gal(K/\bbQ)= \langle \tau \rangle \simeq C_q$.  Recall from Corollary \ref{p-descent_cm} that to each $E/\bbQ$ as above one associates $\Delta$ and $\Delta^+$, and the \'etale algebras $L$ and $L^+$ (respectively) with $L/L^+$ Galois of degree $p-1$. Note that in our case $L^+=K'$. 
We assume that $E/\bbQ$ satisfies $[LK:\bbQ]=2q(p-1)$, which holds generically. 
We want to calculate the
$S$-class group and the $S$-unit group associated with
$LK$ for some finite set of places $S$. As remarked in Remark \ref{deg_too_high}, the degree is still too high to compute the $S$-class groups and $S$-unit groups. 

Therefore, the major contribution of this section is to break down the computation
of $S$-class group and unit group into that of $[K:\bbQ]-1$ number fields of degree $2(p-1)$. 
Assume that $p \ge 5$ and $q | p-1$. Our procedure to compute the $p$-Selmer group for $E/K$ in that case consists of two reduction steps and a key computational step.  
\begin{enumerate}
  \item First reduction step: We use \cite[Theorem 8.1]{schaefer-stoll-p-descent} and Proposition \ref{S-empty}, to reduce $S$ to the primes above $p$ in $LK$ that completely splits. 
  \item Second reduction step: By Theorem \ref{gen_p_descent}
 we can discard the elements in $LK^\times/(LK^\times)^p$ that are represented by elements in $L^\times$ or $K^\times$. 
  \item Computational step: Using the method of norm relations as in \cite{norm_rel} we exhibit $q-1$ subfields $F_1, \ldots F_{q-1}\subset LK$ of degree $2(p-1)$ over $\bbQ$, compute their $S$-class groups and $S$-unit groups, and show that the elements of $\Sel_p(E/K)/\Sel_p(E/\bbQ)$ can be obtained from those.
\end{enumerate}

\begin{remark}\label{remarksteps}\thmbreakheader
\begin{enumerate}
\item Step (1) is specific to the setup of $E/K$ having CM while step (2) follows from the fact that we are only interested in elements of $\Sel_p(E/K)/\Sel_p(E/\bbQ)$. 
    \item If one wants to compute $\Sel_p(E/\bbQ)$, then one only discards elements represented by elements in $K^\times$ in Step (2). However, an extra subfield corresponding to $L$ is added in Step (3).
    \item In practice, one can choose the set $S$ to be empty in Step (1) in most cases. Indeed by Proposition \ref{S-empty} only primes that split completely in $LK/KK'$ remain in $S$.  
\end{enumerate}
\end{remark}

We now detail Step (3).
Recall that $\Gal(L/K')\simeq \bbF_p^\times$ and $q\ |\ p-1$.
Therefore, there is a subfield $F\subset L$ such that $K'\subset F\subset L$ and
$\Gal(L/F)=\langle \sigma \rangle\simeq C_q$. Furthermore, $C_q\times C_q \simeq 
\langle\sigma,\tau\rangle \subset \Aut(LK)$. 
Let $G_i\coloneqq \langle\sigma^i\tau\rangle$ for $0\leq i\leq q-1$ be the 
subgroups of order $q$ inside $\Aut(LK)$, and let $F_i\coloneqq (LK)^{G_i}$, and $F_q
    \coloneqq FK$. Note that $F_0=L$.

Let $S=\{\frak{p}\ |\  \frak{p}\text{ a prime above $p$ in $\cO_{LK}$}\}$, $S_i\coloneqq \{\frak{p}\cap \cO_{F_i}\ |\ \frak{p}\in S\}$ for $0\leq i\leq q$, and $S_F\coloneqq \{\frak{p}\cap \cO_{F}\ |\ \frak{p}\in S\}$.
Let 
\[N_{F_i/F} :\frac{U_{S_i}(F_i)}{U_{S_i}(F_i)^p} \to \frac{U_{S_F}(F)}{U_{S_F}(F)^p}\] be the map
induced by the usual norm map $F_i\to F$. Since $p\ \nmid\ [F_i:F]=q$, the composition $U_{S_F}(F)/U_{S_F}(F)^p \stackrel{\iota}{\to}
U_{S_F}(F_i)/U_{S_F}(F_i)^p \stackrel{N_{F_i/F}}{\to} U_{S_F}(F)/U_{S_F}(F)^p$ is an isomorphism (where $\iota$ is induced by the natural inclusion of $F\subset F_i$). In particular, 
$\iota$ splits the exact sequence 
\[1\to \ker(N_{F_i/F})\to U_{F_i}/U_{S_i}(F_i)^p\to U_{S_F}(F)/ U_{S_F}(F)^p\to 1.\] Therefore \[U_{S_i}(F_i)/U_{S_i}(F_i)^p \simeq U_{S_F}(F)/U_{S_F}(F)^p \oplus\ker(N_{F_i/F}).\] 

Since the imaginary quadratic number field $K'$ satisfies the containement $K'\subset F\subset F_i$ for $0\leq i\leq q$, 
using Dirichlet's unit theorem one computes
    \[\dim_{\bbF_p}(\ker(N_{F_i/F})) =  \frac{p-1}{2}(1-1/q)+\sum\limits_{j\in\{1,2\}}r_{\frak{p}_j,F_i}-r_{\frak{p}_j,F},\] 
    where $r_{\frak{p}_j,-}$
    is the number of prime ideals of $\cO_-$ above the ideal $\frak{p}_j$ of $\cO$ above $p$.
The following proposition shows that elements of $U_S(LK)/U_S(LK)^p$ can be represented by elements of $U_{S_F}(F)$ and $\ker(N_{F_i/F})$ for $0\leq i\leq q$.
    \begin{proposition} \label{unitgp_decomp}
      Let $L$, $K$, $F_i$ and $F$ be as above. Then 
      \[U_S(LK)/U_S(LK)^p \simeq U_{S_F}(F)/U_{S_F}(F)^p\oplus\bigoplus\limits_{i=0}^q \ker(N_{F_i/F}).\]
    \end{proposition}
    \begin{proof}
      We first check that the groups on both sides have the right $\bbF_p$-dimension. The total dimension on the right hand side is
      \begin{align*}
      &\frac{p-1}{2q} -1 + \frac{p-1}{2}(1-1/q)(q+1)+ r_{\frak{p}_1,F}+r_{\frak{p}_2, F}+\sum\limits_{j=1}^2\sum\limits_{i=0}^q\left(r_{\frak{p}_j,F_i}-r_{\frak{p}_j,F}\right)\\ &=\frac{q(p-1)}{2}-1 + r_{\frak{p}_1,F}+r_{\frak{p}_2, F}+ \sum\limits_{j=1}^2\sum\limits_{i=0}^q\left(r_{\frak{p}_j,F_i}-r_{\frak{p}_j,F}\right).
      \end{align*}
      If a prime $\frak{q}$ of $\cO_{F}$ splits in $\cO_{F_i}$, one has $r_{\frak{p}_j, F_i}=qr_{\frak{p}_j, F}$. If it does not, one has
      $r_{\frak{p}_j, F_i}=r_{\frak{p}_j, F}$. Furthermore, if $r_{\frak{p}_j, F_i}=qr_{\frak{p}_j, F}$ for two distinct $i$, then $r_{\frak{p}_j, F_i}=qr_{\frak{p}_j, F}$ for all $i$ and $r_{\frak{p}_j,LK}=q^2r_{\frak{p}_j, F}$. In this case one obtains
      \[r_{\frak{p}_j,F}+\sum\limits_{i=0}^q\left(r_{\frak{p}_j,F_i}-r_{\frak{p}_j,F}\right)=r_{\frak{p}_j,F}+(q^2-1)r_{\frak{p}_j, F}=q^2r_{\frak{p}_j, F}=r_{\frak{p}_j, LK}.\]
      Similarly, one checks that the equality $r_{\frak{p}_j,F}+\sum\limits_{i=0}^q\left(r_{\frak{p}_j,F_i}-r_{\frak{p}_j,F}\right)=r_{\frak{p}_j, LK}$ holds when $r_{\frak{p}_j,F_i}=qr_{\frak{p}_j, F}$ for exactly one $i$, and when $r_{\frak{p}_j,F_i}=r_{\frak{p}_j, F}$. Therefore, the $\bbF_p$-dimension matches on both sides.
      
      If each of the
      groups in the summand on the right-hand side intersects the other trivially, we are
      done. To see this, note that if $u\in \ker(N_{F_i/F}\cap N_{F_j/F})$ with 
      $i\ne j$, then $u$ can be represented by an element in $U_{S_F}(F)$ but then $u\ne
      \ker(N_{F_i/F})$. Lastly, $U_{S_F}(F)/U_{S_F}(F)^p$ intersects trivially with each
      $\ker(N_{F_i/F})$.
    \end{proof}

Now we look at the class group contribution. The arguments are  similar as
for the unit groups. We once again have the following exact sequence
\[0\to \Cl(F_i)[p]\to \Cl(LK)[p]\stackrel{N_{LK/F_i}}{\to} \Cl(F_i)[p]\to 0,\] where $N_{LK/F_i}$ is the usual norm map on ideal classes. 
We obtain a similar sequence when $F_i$ is replaced by $F$. The following
proposition shows
that $\Cl(LK)[p]$ is generated by $\Cl(F_i)[p]$ and $\Cl(F)[p]$. 

\begin{proposition}\label{ClassGroupreduceF}
  Let $\frak{p}$ be an ideal in the maximal order $\cO_{LK}$ of $LK$ representing a non-trivial 
  element in
  $\Cl(LK)[p]$. Then there is a field $F_i$ for $0\leq i\leq q$ such that $N_{LK/F_i}([\frak{p}])\ne 0$ or $N_{LK/F}([\frak{p}])\ne 0.$ 
  
  In particular, we obtain
  that $[\frak{p}]$ is in the image of $\Cl(F_i)[p]\to \Cl(LK)[p]$ or $\Cl(F)[p]\to \Cl(LK)[p]$.  
\end{proposition}
\begin{proof}
  Assume that for every field $F_i$ and $F$, $N_{LK/F_i}([\frak{p}])$ and
  $N_{LK/F}([\frak{p}])$ are trivial. This implies that 
  \[[\frak{p}]\in\bigcap\limits_{i=0}^q\ker(N_{LK/F_i}),\] and hence that
  \[\prod\limits_{i=0}^q \prod\limits_{g\in G_i}\frak{p}^g\] is a principal ideal. But 
  \[\prod\limits_{i=0}^q \prod\limits_{g\in G_i}\frak{p}^g=
  \frak{p}^q\prod\limits_{g\in \Gal(LK/F)} \frak{p}^g=\frak{p}^qN_{LK/F}(\frak{p})\]
 implies that $\frak{p}^q$ is a principal ideal since $N_{LK/F}([\frak{p}])$ is principal. However, the
  order of $[\frak{p}]$ is $p$, hence $\frak{p}$ is a principal ideal. 
\end{proof}
\subsection{Galois module structure of $\Sha(E/K)[p]$}\label{subgaloismodule}
In this section we consider the action of $\Gal(K/\bbQ)$ on $\Sha(E/K)[p]$. Recall from Section \ref{classgp_unitgp} that $\Gal(K/\bbQ)=\langle \tau\rangle \iso C_q$, where $q\ |\ p-1$ (unless stated otherwise). It follows that each $\Sha(E/K)[p]_{\theta}$ as defined in Section \ref{sec:conjGalSha}, is a 1-dimensional $\bbF_p$-representation with $h_\theta=t-\alpha_\theta$, where $\alpha_\theta$ is a primitive $q$th root of unity in $\bbF_p$. Any element $s\in(\Sha(E/K)[p])_\theta$ satisfies $s^\tau=\alpha_\theta s$. For a given $s\in (\Sha(E/K)[p]_\theta)$, the following proposition computes an $1\leq i\leq q-1$ such that $s$ is represented by an element in $F_i$.

\begin{proposition}\label{U(LK)U(Fi)}
  Let $\phi: \Gal(L/K')\to
  \bbF_p^\times$ be the group homomorphism  defined in Theorem \ref{gen_p_descent}.
  We have
  \[U_S(LK)/U_S(LK)^p\cap\frac{\Sel_p(E/K)}{\Sel_p(E/\bbQ)}=\bigoplus\limits_{\theta}U_S(LK)/U_S(LK)^p\cap\left(\frac{\Sel_p(E/K)}{\Sel_p(E/\bbQ)}\right)_\theta\subset\bigoplus\limits_{i=1}^{q-1}\ker(N_{F_i/F}).\]
  Moreover, if $u\in \left(\frac{\Sel_p(E/K)}{\Sel_p(E/\bbQ)}\right)_\theta$ with $\alpha_\theta=\phi(\sigma)^{-i}$, then
  $u$ can be represented by an element in $\ker(N_{F_i/F})$. 
\end{proposition}
\begin{proof}
  Recall that $\langle \sigma\rangle =\Gal(L/F)$ and $G_i=\Gal(F_i/F)=\langle \sigma^i\tau\rangle$ for $1\leq i\leq q-1$. We note that each of the $\ker(N_{F_i/F})$ is stable under the action of the image of $\sigma$ in $\langle\sigma,\tau\rangle/G_i \simeq \Gal(F_i/F)$ and that by definition, $\sigma$ acts trivially on
$\ker(N_{F_q/F})$ and on $U(F)/U(F)^p$. This along with Proposition \ref{unitgp_decomp} implies the first part of the proposition. 
  Now if $u$ as above represents an element in $p$-Selmer group then
  $u^\sigma=u^{\phi(\sigma)}$, which along with $u^\tau = \alpha_\theta\tau$ proves the result. 
\end{proof}

The next proposition shows that if $E/\bbQ$ has good reduction at $p$, then we can further restrict $S$.
\begin{proposition}\label{reduce-S}
Let $E/\bbQ$ and $K$ be as above. If $E/\bbQ$ has good reduction at $p$, then exactly one of the primes in $\cO$ above $p$ (say $\frak{p}_1$) is unramified and the other one (say $\frak{p}_2$) is totally ramified  in $\cO_L$ and hence in $\cO_F$.
Furthermore, one can choose $S=\{\frak{p}\ |\ \frak{p}\text{ an prime ideal in $\cO_{LK}$ above $\frak{p}_1$}\}$ in order to compute $\Sel_p(E/K)$.
\end{proposition}
\begin{proof}

Since $E/\bbQ$ has good reduction at $p$ and $p$ splits in $\cO$, $E$ has ordinary reduction at $p$. Therefore, exactly one of the primes (say $\frak{p}_2$) ramifies in $\cO_L$. If $\frak{p}_2$ is not totally ramified, then the ramification index of a prime above $\frak{p}_2$ in $\cO_{\bbQ(E[p])}< p-1$ but $\bbQ(\mu_p)\subset \bbQ(E[p])$.

The ramification index of a prime $\frak{q}$ of $\cO_{KK'}$ above $\frak{p}_2$ in $\cO_{LK}$ is either $(p-1)/q$ or $(p-1)$, depending on whether $p$ ramifies in $\cO_K$ or not. In either case, $\frak{q}$ does not split in $\cO_{LK}$ and therefore by Proposition \ref{S-empty} we are done. 
\end{proof}

\begin{remark} \label{U(LK)U(F)_remark}
 Proposition \ref{U(LK)U(Fi)} can be used to restrict the search for $p$-Selmer elements from $U(LK)$ to $U(F_i)$. One saves a factor of $q$ in the degree of the extensions considered. Moreover, if $\frak{p}_1$ as in Proposition \ref{reduce-S} does not split in $F_i$, then one can choose $S_i$  to be empty while computing the contribution to $\Sel_p(E/K)$ from $F_i$.
The next remark shows how to search for elements in $\Sel^{(p)}(E/K)$ coming from $U_{S_i}(F_i)$ in practice. 
\end{remark}

\begin{remark}\label{relavent-rank_unit_group}
Since $ \frac{\pair{\tau}{\gamma}}{G_i}\simeq \Gal(F_i/K')\simeq \bbF_p^\times$, one obtains a representation of $C_{p-1}$ on $\frac{U_{S_i}(F_i)}{U_{S_i}(F_i)^p}$.  
Theorem \ref{gen_p_descent} implies that a $p$-Selmer element represented by an element in $U_{S_i}(F_i)$ must belong to the $\phi_i$-eigenspace of $\frac{U_{S_i}(F_i)}{U_{S_i}(F_i)^p}$, where $\phi_i: \Gal(F_i/K')\to \bbF_p^\times$ is induced by the morphism $\phi$ in Theorem \ref{gen_p_descent}.
To obtain the $\phi_i$-eigenspace, one could randomly choose an element $u$ in $\frac{U_{S_i}(F_i)}{U_{S_i}(F_i)^p}$ and compute the subspace generated by $u$: $U\coloneqq \{\gamma_i^j(u)\ |\ 0\leq j\leq p-2\ \}$, where $\gamma_i$ is the image of $\gamma$ in $\frac{\pair{\tau}{\gamma}}{G_i}$. Since the action of $\gamma_i$ is easy to compute on $U$, one can rather easily check using linear algebra if $\omega(\gamma_i)=\omega(\gamma)$ is an eigenvalue of the matrix corresponding to the action of $\gamma_i$ on $U$.
\end{remark}

In the light of Conjecture \ref{exact_question} and our assumption that $q\ |\  p-1$ one must have $\theta$ as a 1-dimensional representation. If   $h_\theta(t)=t-\alpha_\theta$, then the irreducible representation $\theta'$ with associated polynomial $h_{\theta'}(t)= t-\alpha_\theta^{-1}$ also appears. This is because the ideal $(\sL(E,\chi))$ is invariant under the complex conjugation. 
The following proposition proves this independently of any conjecture. 

\begin{proposition}\label{ctp_consequence}
    Let $h_\theta(t)=t-\alpha_\theta$ and $K$ be as above and assume that $\Sha(E/K)[p^\infty]=\Sha(E/K)[p]$. 
    Then there is an irreducible subrepresentation $\theta'$ such that $h_{\theta'}(t)=t-\alpha_\theta^{-1}$ which appears with same multiplicity as $\theta$.
\end{proposition}
\begin{proof}
Let $\pair{\cdot}{\cdot}_\ctp: \Sha(E/K)[p]\to \bbZ/p\bbZ$ be the Cassels-Tate pairing. Then $\pair{\cdot}{\cdot}_\ctp$ is alternating with the following properties
\begin{enumerate}
\item If $\forall\ b\in \Sha(E/K)[p]$, $\pair{a}{b}=0$  $\iff$ $a\in\Im(\Sha(E/K)[p^2]\stackrel{\cdot p}{\to}\Sha(E/K)[p])$. 
\item $\pair{a^\tau}{b^\tau}_\ctp=\pair{a}{b}_\ctp$, where  $\Gal(K/\bbQ)=\langle\tau\rangle$. 
\end{enumerate}
For details on the Cassels-Tate pairing and various known definitions see \cite{cassels}.

As a $\Gal(K/\bbQ)$-representation,
$a^\tau=\alpha_\theta a$ for all $a\in (\Sha(E/K)[p])_\theta$. 
Let $\theta_2$ and $\theta_1$ be two distinct subrepresentations with generators $a_1$ and $a_2$. Using the second property of $\pair{\cdot}{\cdot}_\ctp$ above one obtains 
\[\pair{a_1^\tau}{a_2^\tau}_\ctp=\alpha_{\theta_1}\alpha_{\theta_2}\pair{a_1}{a_2}_\ctp=\pair{a_1}{a_2}_\ctp.\]
Therefore $\pair{a_1}{a_2}_\ctp=0$ $\iff$ $\alpha_{\theta_1}\ne \alpha_{\theta_2}^{-1}.$
The assumption $\Sha(E/K)[p^\infty]=\Sha(E/K)[p]$ implies that for each $\theta$, there is a $\theta'$ with $\alpha_{\theta'}=\alpha_\theta^{-1}$.

We now prove the multiplicity part of the claim. Let $\alpha\in \mu_q\subset \bbF_p$ and define \[M_{\alpha}\coloneqq\bigoplus\limits_{\substack{\theta\\h_\theta(t)=t-\alpha}}(\Sha(E/K)[p])_\theta.\]
Note that $\dim_{\bbF_{p}}(M_\alpha)$ counts the multiplicity of $\theta$.
Without loss of generality,  assume that $\dim_{\bbF_p}(M_\alpha)>\dim_{\bbF_p}(M_{\alpha^{-1}})$. Then there is a $a\in M_\alpha$ such that $\pair{a}{b}_\ctp=0$ for all $a'\in M_{\alpha^{-1}}$. Hence $\pair{a}{b}=0$
for all $b \in \Sha(E/K)[p]$. Therefore, $\dim_{\bbF_p}(M_\alpha)=\dim_{\bbF_p}(M_{\alpha^{-1}})$. This proves the proposition.
\end{proof}
\begin{remark}
   In the light of Remark \ref{U(LK)U(F)_remark}, the above proposition reduces our search for $p$-Selmer elements coming from $U(LK)$ to searching in $U(F_i)$ for $1\leq i\leq (q-1)/2$.  
\end{remark}

\begin{remark}
Let $\bbF\coloneqq \bbF_p(\zeta_q)$.
In the case when $q\ \nmid\ p-1$ one can extend the pairing $\pair{\cdot}{\cdot}_\ctp$ to 
\[\pair{\cdot}{\cdot}: (\Sha(E/K)[p])\otimes_{\bbF_p}\bbF\times (\Sha(E/K)[p])\otimes_{\bbF_p}\bbF\to \bbF.
\]
Since $\bbF$ is a trivial $\Gal(K/\bbQ)$-module, it follows that if $(\Sha(E/K)[p])_\theta\ne 0$ with the associated polynomial $h_\theta$ then 
$(\Sha(E/K)[p])_{\theta'}\ne 0$, where the polynomial $h_{\theta'}$ has the property that for each root $\alpha$ of $h_{\theta'}$, $\alpha^{-1}$ is a root of $h_{\theta}$.
Furthermore, $\theta$ and $\theta'$ occur with same multiplicity.
\end{remark}

Algorithm \ref{alg:cap} summaries the above discussion.

\begin{algorithm}[H]
\caption{}\label{alg:cap}
\begin{algorithmic}
\Require \thmbreakheader\begin{enumerate} \item Elliptic curve $E/\bbQ$ with complex multiplication by order $\cO$ and a prime $p$ such that $p$ splits in $\cO$.
\item A cyclic extension $K/\bbQ$ of prime degree $q$ such that $q\ |\ p-1$.
\end{enumerate}
\Ensure The minimal polynomials $h_\theta$ and the eigenspaces $\left(\Sel_p(E/K)\right)_\theta$.
\State \thmbreakheader
\begin{enumerate}
\State Compute the \'{e}tale algebra $L$ for $E/\bbQ$ by adjoining a point corresponding to a root of the $p$-torsion polynomial, the subfield $F\subset L$ with $[L:F]=q$, $\Gal(L/F)=\langle \sigma\rangle$ and $\Gal(K/\bbQ)=\langle\tau\rangle$. 
\State For each prime $\frak{p}\subset \cO_K$ above $p$, compute the group $\prod\limits_{\frak{p}|p}\frac{J(K_\frak{p})}{pJ(K_\frak{p})}$.
\State Choose at random  primitive elements $l\in L$ and $k\in K$, and compute the elements $\alpha_i=\sum\limits_{j=0}^{q-1}\sigma^{ij}(l)\tau^j(k)$ for $0\leq i\leq q-1$. 
\If{ for each $i$ such that  $0\leq i\leq q-1$, $\deg(\alpha_i)$ over $F$ is $q$}
\State $F_i\gets F(\alpha_i)$ and $G_i\gets\langle \sigma^i\tau\rangle$.
\EndIf
\State Compute the maximal orders $\cO_{F_i}$ using equation order for $F_i$ and $\cO_F$.
\For{$0\leq i\leq q-1$ (one can let $1\leq i\leq q-1$ in order to compute $\left(\frac{\Sel_p(E/K)}{\Sel_p(E/\bbQ)}\right)_\theta$)}
\State Compute prime ideals $\frak{p}\subset \cO$ above $p$
that split in $\cO_{F_i}$.
\State $S_i\gets \{\text{primes of }\cO_{F_i}\text{ above $\frak{p}$}\}.$
\State Compute the $S_i$-class groups $\Cl_{S_i}(\cO_{F_i})$ and unit groups $U_{S_i}(\cO_{F_i})$ of $\cO_{F_i}$. 
\State Compute $R(F_i, S_i;p)$ using $\Cl_{S_i}(\cO_{F_i})$ and $U_{S_i}(\cO_{F_i})$ of $\cO_{F_i}$.
\State Use Remark \ref{relavent-rank_unit_group} to search for $p$-Selmer elements and let $H_i$ be the subgroup of the $\Sel_p(E/K)$ found. 
\State $h_\theta\gets  (x-\phi(\sigma)^{-i})^{\dim_{\bbF_p}(H_i)}$ and $\left(\frac{\Sel_p(E/K)}{\Sel_p(E/\bbQ)}\right)_\theta\gets H_i$.
\EndFor
\end{enumerate}
\end{algorithmic}
\end{algorithm}

\begin{remark}
     In order to verify Conjecture \ref{exact_question} we will only consider the curves with rank 0 over $K^\chi$ and $\bbQ$, therefore the above algorithm will actually be computing $\left(\frac{\Sha(E/K^\chi)[p]}{\Sha(E/\bbQ)[p]}\right)_\theta$. Furthermore, if we are ready to assume the BSD-conjecture and if $\Sha_\an(E/K^\chi)[p^\infty]=p^2=\Sha(E/K^\chi)[p]$, then using Proposition \ref{ctp_consequence} one can reduce the search from $1\leq i\leq q-1$ to $1\leq i\leq \frac{q-1}{2}$ in step 8 of the above algorithm.  Moreover, if $E$ has good reduction at $p$ and there is exactly one prime above $p$ in $\cO_K$, then using Proposition \ref{reduce-S} one reduces the computation of $R(F_i,S_i;p)$ for those values of $i$ when $S_i=\emptyset$.
\end{remark}
 \subsection{Numerical evidence} We present an example of the above procedure for the curve with LMFDB-label ``7056.bg1''. More examples can be found in \cite{github}, together with the codes.  

Fix an embedding of $\bbZ[\zeta_5]\hookrightarrow\bbC$ by $\zeta_5\mapsto \exp^{4i\pi /5}.$
 Let $K^\chi:= \bbQ(\zeta_{11})^+$, $p=11$ and 
$E$ given by 
$y^2=x^3-262395x +51731946$. Let $\chi$ be a primitive character with field of definition $K^\chi$, conductor $\frak{f}_\chi=11$, order $d=5$ and root number $w(\chi)=1$.
Let $\tau:K^\chi\to K^\chi$ be the automorphism that sends a root $\alpha$ of the polynomial $x^5+x^4-4x^3-3x^2+3x+1$ to $\alpha^2-2$. 
We verify that Conjecture \ref{exact_question} holds for $E$ for all characters of $K^\chi$.

    $E$ has complex multiplication by $\bbZ[\sqrt{-7}]$. Therefore, $K'=\bbQ(\sqrt{-7})$. $L$ is the field of definition of one of the $11$-torsion points (say $P$) of $E/\bbQ$ that we fix and $[L:\bbQ]=20$. 
    Following notations in Theorem \ref{gen_p_descent} we fix a generator $\gamma$ of $\Gal(L/K')$ and compute $\phi(\gamma)=-5$.
    One computes the subfields $F_i$ as in step 5 in Algorithm \ref{alg:cap} using $L$ and $K^\chi$ and $\sigma =\gamma^2$.
    Exact description of the fields can be found in  \cite{github}. 
    
We check using Magma that $E$ satisfies the conditions of  Theorem \ref{MainTh}. We find that 

$\rk(E/\bbQ)=\rk(E/K^\chi)=0$ (using \texttt{TwoSelmerGroup()}). This implies that $\Sel_{11}(E/K^\chi)=\Sha(E/K^\chi)[11]$, and
 $|\Sha_{\an}(E/\bbQ)|=1$  and $\Sha_{\an}(E/K^\chi)\simeq \bbF_{11}^2\oplus \bbF_{31}^2$ (using \texttt{ConjecturalSha()}). 
 Furthermore, \[\prod\limits_{M_{K^\chi}}c_v(E)=\prod\limits_{M_{\bbQ}}c_v(E)\] for $i\in\{1,2\}$, where $M_{K^\chi}$ is the set of finite places of $K^\chi$,
and $E(\bbQ)_\tors=E(K^\chi)_\tors\iso \bbZ/2\bbZ$.

By Theorem \ref{MainTh} we have 
$N_{\bbQ(\zeta_5)^+/\bbQ}(\sL(E,\chi))=11\cdot 31$. Using the fact that $31$ is principal in $\cO_{\bbQ(\zeta_5)^+}$, we compute that 
$(\sL(E, \chi))_{11} = \frak{p}_1 \overline{\frak{p}_1}$, with $\frak{p}_1 = ((\zeta_{5}-3),11)$, where we have pulled back $\sL(E,\chi)$ under the previously fixed embedding $\bbZ[\zeta_5]\to \bbC$.

One checks that $11$ is a prime of good reduction of $E$, therefore using Proposition \ref{reduce-S} exactly one of the two ideals of $\cO$ above $11$ (say $\frak{p}$) is unramified. Moreover, since $11$ is totally ramified in $K^\chi$, one deduces that one can choose $S_i=\emptyset$ as in step 10 in Algorithm \ref{alg:cap}.
By Proposition \ref{ClassGroupreduceF} we reduce our computations to unit groups as $\Cl(F_i)[11]=0$ for all $1\leq i\leq 4$ (we use the function \texttt{ClassGroup()}).  
By Propositions \ref{U(LK)U(Fi)} and  \ref{relavent-rank_unit_group}  we can find a rank $1$ submodule generated by $u_i$ of $R(F_i,\emptyset;11)$ for all $i$, where $R(F_i,\emptyset;11)$ are computed using \texttt{pSelmerGroup()}. By Proposition \ref{ctp_consequence} we can further refine the search to $U(F_i)$ for $1\leq i\leq 2$. Since $|\Sha(E/K)[11]|=11^2$, Proposition \ref{ctp_consequence} also shows that only one of $u_1$ or $u_2$ represents an element $\Sel_{11}(E/K^\chi)[11]$. We find that $u_1\in \Sel_{11}(E/K^\chi)$.
By Proposition \ref{ctp_consequence} there is $u_4\in U(F_4)$ such that $u_4\in \Sel_{11}(E/K^\chi)$. Using Proposition \ref{U(LK)U(Fi)} we have $h_\theta$ is either $x-3$ or $x-4$. This shows that $(h_\theta(\zeta_5),11)=\frak{p}_1$ or $\overline{\frak{p}_1}$.

 Performing the same procedure as above for $E/\bbQ$ with LMFDB-label ``6400.a1'' and  $\chi$ as above with the same embedding of $\bbZ[\zeta]\hookrightarrow\bbC$, one obtains $(\sL(E,\chi))_{11}=\frak{p}_2\overline{\frak{p}}_2$, and $h_\theta$ is either $x-5$ or $x-9$.
This shows that as $E$ and the Galois-module structure of $\Sha(E/K)[p]$ varies, so does the factorization of $(\sL(E,\chi))_p$.

\begin{remark}
    Tommy Hofmann has computed the full $\Cl(LK^{\chi})$ in the above case and his computations show that $\Cl(LK^\chi)\simeq (\bbZ/2\bbZ)^7$.
\end{remark}

\begin{remark}
Let $E/\bbQ$ be the elliptic curve with LMFDB-label ``207025.by1'' given by the Weierstrass model
\[E/\bbQ: y^2 + xy = x^3 - x^2 - 7698742x + 
8223502041,\]
and $K^\chi= \bbQ(\zeta_{11})^+$. Then $\Sha(E/\bbQ)=0$ and the conjectural order of $|\Sha(E/K^\chi)|$  is $=11^4$.
In the notation of Theorem \ref{MainTh}
the ideal $(\sL(E,\chi))=(\frak{p}_2\overline{\frak{p}_2})^2$, where $\frak{p}_2=(\zeta_5-3,11).$ 
Hence, if the Conjecture \ref{exact_question} holds, then Proposition \ref{U(LK)U(Fi)}
along with the fact that $\Cl(F_i)[11]=0$ implies that we must have $\Sha(E/K^\chi)[11]\simeq (\bbZ/11^2\bbZ)^2$ (as an abelian group). In fact, we compute $\Sha(E/K^\chi)[11^2]\simeq (\bbZ/11\bbZ)^2$.
\end{remark}

\begin{remark}\label{no_isogeny_descent}
Our idea to perform an 11-descent on CM elliptic curves in order to produce numerical evidence towards Conjecture \ref{exact_question} is optimal in the sense that finding elliptic curves $E/\bbQ$ that admit an 11-isogeny seems to be a hard problem. Given an order $5$ Dirichlet character $\chi$, this is equivalent to finding $K^\chi$-rational points on the elliptic curve $X_0(11)$ such that the $j$-invariant is rational. A simple search through $C_5$ number fields $K^\chi$ on LMFDB with ``low'' discriminant yields no field where $X_0(11)$ even acquires a $K^\chi$-rational point.

Another idea could have been to use $\bbQ$-rational points on $X_0(13)$, which is a rational curve, in order to create examples for $p=13$ and $q=7$. However, not much can be said about the class group and the unit group of the degree 84 \'etale algebra corresponding to 13-isogeny descent in this case. 
\end{remark}    

We end this section with the following list of curves for which we have verified Conjecture \ref{exact_question} for all non-trivial primitive Dirichlet characters factoring through the field $K=\bbQ(\zeta_{11})^+$ with $\Gal(K/\bbQ)$ generated by $\tau$ as before, and $p=11$.
The predicted order of $\Sha(E/K)[11]$ is $11^2$ in each case, and in fact the curves in the following list along with ``7056.bg1'' and ``207025.by1'' are all the curves in LMFDB with conductors $\frak{f}_E$ coprime to 11 and $\Sha(E/K)[11]\ne 0$ that satisfy the conditions of Algorithm \ref{alg:cap}.

\vspace{.2in}
\begin{center}
\begin{tabular}{|c|c|c|}
    \hline
    LMFDB-label &CM-discriminant &$h_\theta$ \\
    \hline
    \hline
    5776.i1 & -19 & $(x-5),\ (x-9)$\\
     \hline
    6400.r1 & -43 & $(x-5),\ (x-9)$\\
     \hline
     16641.g1 & -43 &
     $(x-5),\ (x-9)$\\
     \hline
     57600.ch2& -8 & $(x-3),\ (x-4)$\\
     \hline
     90601.c1 & -43 & $(x-3),\ (x-4)$\\
     \hline
     215296.c1& -8 & $(x-3),\ (x-4)$\\
     \hline
     461041.h1& -28 & $(x-5),\ (x-9)$\\
     \hline
     499849.d1& -28 & $(x-5),\ (x-9)$\\    
    \hline
\end{tabular}
\end{center}

\end{document}